\documentclass[preprint,sort&compress,12pt]{elsarticle}
\usepackage{amscd,amssymb,amsmath,amsthm}
\usepackage[all]{xy}
\usepackage{array}
\usepackage{etoolbox}
\usepackage{mathtools}
\DeclarePairedDelimiter{\ceil}{\lceil}{\rceil}
\DeclarePairedDelimiter{\floor}{\lfloor}{\rfloor}
\makeatletter
\patchcmd{\ps@pprintTitle}{\footnotesize\itshape
        \hfill\today}{\relax}{}{}

\newdir{ >}{!/8pt/\dir{}*\dir{>}}
\newtheorem*{thmA}{Theorem A}
\newtheorem*{thmB}{Theorem B}
\newtheorem*{thmC}{Theorem C}
\newtheorem*{thmD}{Theorem D}
\newtheorem*{thmE}{Theorem E}
\newtheorem*{thmF}{Theorem F}
\newtheorem{theorem}{Theorem}[section]

\newtheorem{lemma}[theorem]{Lemma}

\newtheorem*{con7*}{Conjecture 7*}

\theoremstyle{definition}
\newtheorem{definition}[theorem]{Definition}

\newtheorem{conj}{Conjecture}
\DeclareMathOperator{\e}{\mathit{exp}}
 \DeclareFontFamily{U}{wncy}{}
    \DeclareFontShape{U}{wncy}{m}{n}{<->wncyr10}{}
    \DeclareSymbolFont{mcy}{U}{wncy}{m}{n}
    \DeclareMathSymbol{\Sh}{\mathord}{mcy}{"58} 

\journal{}

\begin{document}

\begin{frontmatter}

\title{ On Schurs exponent conjecture and its relation to Noether's Rationality problem}

\author[IISER TVM]{V.Z. Thomas\corref{cor1}}
\address[IISER TVM]{School of Mathematics,  Indian Institute of Science Education and Research Thiruvananthapuram,\\695551
Kerala, India.}
\ead{vthomas@iisertvm.ac.in}
\cortext[cor1]{Corresponding author. \emph{Phone number}: +91 8921458330}

\begin{abstract}
In this short survey article, we aim to provide an up to date information on the progress made towards Schurs exponent conjecture and related conjectures. We also mention the connection between Schurs exponent conjecture and Noether's Rationality problem.
\end{abstract}


\end{frontmatter}

 \section{Introduction}
 
 Let $\e(G)$ denote the exponent of the group $G$, which is the smallest positive integer $n$ such that $g^n=1$ for all $g\in G$, and let $H_2(G, \mathbb{Z})$ denote the second homology group with coefficients in $\mathbb{Z}$. Schur's exponent conjecture can now be stated as
 
  \begin{conj}\label{C1}
 if $G$ is a finite group, then $\e(H_2(G,\mathbb{Z})) \mid \e(G)$.  
\end{conj}
 In \cite{MRJ}, the author mentions the above as a conjecture. So this conjecture is at least 46 years old. In \cite{BKW}, the authors found a remarkable counterexample to this conjecture. Their counterexample involved a $2$-group of order $2^{68}$ with $\e(G)=4$ and $\e(H_2(G,\mathbb{Z}))=8$. In \cite{PM1}, the author mentions another counterexample to Schur's conjecture. His counterexample involves a $2$-group of order $2^{11}$ with $\e(G)=4$ and $\e(H_2(G,\mathbb{Z}))=8$. All the counterexamples produced so far involve $2$ groups. During the write up of this survey, the author was working on groups of exponent $5$. The author could manage to prove the conjecture for $R(2,5)$ and $R(3,5)$, where $R(d,5)$ denotes the largest finite quotient of the $d$ generator Burnside group $B(d,5)$ of exponent $5$. For $R(4,5)$, it is known that the nilpotency class is at most 24 (cf. \cite{HNV}). I felt, I could prove Schur's conjecture if the nilpotency class of $R(4,5)$ was less than 24, but I only managed to prove that $\e(H_2(G,\mathbb{Z})) \mid  (\e(G))^2$. By this point, I had started to feel that the conjecture may not be true for $R(d,5)$, where $d>4$. To decide whether to continue my approach further or not, I contacted M. R. Vaughan-Lee asking him if the nilpotency class of $R(4,5)$ is less than 24. He replied saying that it is an interesting question, but it is not known and perhaps very difficult to decide, but most probably it is less than 24. I explained to him why I was interested in the nilpotency class of $R(4,5)$ and its connection to Schur's conjecture and mentioned that one could expect a counterexample for groups of exponent 5 and shared a preliminary version of this survey article with him. He got interested in the problem and found a remarkable counterexample to Schur's conjecture. His counterexample involves a $4$ generator group $G$ of exponent $5$, nilpotency class 9, but the $\e(H_2(G,\mathbb{Z}))=25$. The order of this group is $5^{4122}$. The group $G$ is defined as follows. He first defines a four generator group $H$ with the presentation $<a,b,c,d|[b,a]=[d,c]>$. Then he defines $G$ to be the largest quotient of $H$ with exponent $5$ and nilpotency class $9$. He also obtained another counterexample of a $4$ generator group of order $3^{11983}$, nilpotency class $9$ and exponent $9$, with $\e(H_2(G,\mathbb{Z}))=27$.
 
 Much before these counterexamples were produced by Vaughan-Lee, the author of  \cite{PM5} conjectures the following
 
 \begin{conj}\label{C2}
 if  G is a finite group, then $\e(H_2(G,\mathbb{Z})) \mid  (\e(G))^2$.  
\end{conj}
 
 The authors of \cite{APT} conjecture that,
 
 \begin{conj}\label{C3}
 if  G is a finite p-group, then $\e(H_2(G,\mathbb{Z})) \mid p\ \e(G)$.  
\end{conj}
The authors of \cite{APT} conjecture this at least 6 months prior to the counterexample produced by Vaughan-Lee. Clearly the counterexamples for Schur's conjecture given by the authors of \cite{BKW} and \cite{PM1} are not counterexamples for Conjecture \ref{C3}. The counterexamples produced by Vaughan-Lee to Schur's conjecture are not counterexample to Conjecture \ref{C3}. If Conjecture \ref{C3} holds, then using a standard argument given in Theorem 4, Chapter IX of \cite{JPS}, it will follow that $\e(H_2(G,\mathbb{Z})) \mid (\e(G))^2$ for any finite group $G$.\\

For the benefit of the readers, without further ado, we collect the classes of groups for which the above conjectures have been proved. After that we will give a brief outline of the strategy used to prove the above conjectures. 

\begin{thmA}
\begin{itemize}
Conjecture \ref{C1} holds for the following classes of groups :
\item[$1.$] Groups for which the Bogomolov multiplier is trivial. ( cf. \cite{KK} for $\Sh$-rigid groups and triviality of Bogomolov multiplier for such groups.)
\item[$2.$] If the commuting probability of a finite group $G$ is greater than 0.25. (cf. \cite{JM})
\item[$3.$] Central extensions of most finite simple groups. (cf. \cite{DLT1})
\item[$4.$] $p$-groups of nilpotency class less than or equal to $3$ for all primes $p$. (cf. \cite{PM2})
\item[$5.$] Powerful $p$-groups. (cf. \cite{LM}, \cite{APT})
\item[$6.$] $p$-groups of maximal class. (cf. \cite{PM6})
\item[$7.$] Potent $p$-groups (cf. \cite{PM4}, \cite{APT})
\item[$8.$] $p$-groups of nilpotency class less than or equal to $5$ for an odd prime $p$. (cf. \cite{APT})
\item[$9.$] $p$-groups of nilpotency class less than or equal to $p$ for an odd prime $p$. (cf. \cite{APT})
\item[$10.$] Metabelian $p$-groups of exponent $p$.  (cf. \cite{PM1}, \cite{APT})
\item[$11.$] $p$-central metabelian $p$ groups. (cf. \cite{APT})
\item[$12.$] An odd $p$-group with an abelian Frattini subgroup. (cf. \cite{AT})
 \item[$13.$] A $p$-group such that the commutator subgroup of $G$ is cyclic. (cf. \cite{AT})
 \item[$14.$] Finite metacyclic groups. (cf. \cite{AT})
 \item[$15.$] Abelian by cyclic groups. (cf. \cite{AT})
 \item[$16.$] $3$-Engel groups. (cf. \cite{PM7})
 \item[$17.$] $4$-Engel groups of exponent $e$, where $e$ is not divisible by $2$ or $5$. (cf. \cite{PM7})
\item[$18.$] A group $G$ with $\e(G/Z(G)) = p$ and $p \in \{2,3\}$.  (cf. \cite{AT})
\item[$19.$] A finitely generated group $G$ such that $\e(G/Z(G)) = 6$. (cf. \cite{AT})

\end{itemize}
 \end{thmA}
 
 \begin{thmB}
 \begin{itemize}
 Conjecture \ref{C2} is valid for the following classes of groups:
\item[$1.$]  A group of exponent $e$, where $e\leq 7$. (cf. \cite{NS2})
\item[$2.$] An odd solvable group of exponent $p$ and derived length $3$. (cf. \cite{AT})
\item[$3.$]  A $p$-group having an abelian normal subgroup $N$ of index $p^l$, $p \neq 2$ and $l$ is less than max $\{7,p+2\}$. (cf. \cite{AT})
 \item[$4.$] An odd $p$-group whose commutator subgroup is powerful. (cf. \cite{AT})
  \item[$5.$] An odd $p$-group $G$ such that $\gamma_{p+1}(G)$ is powerful. (cf. \cite{AT})
\end{itemize}
 \end{thmB}
 
 \begin{thmC}
 \begin{itemize}
 Conjecture \ref{C3} holds for the following classes of groups :
\item[$1.$] An odd $p$-group of nilpotency class less than or equal to $7$. (cf. \cite{AT})
\item[$2.$]  A $p$-group $G$ such that $\e(Z(G))=p$ and nilpotency class of $G$ is at most $p+1$. (cf. \cite{AT})
\item[$3.$]  A group having an abelian normal subgroup $N$ of index $p^2$, and $p \neq 2$. (cf. \cite{AT})
 \item[$4.$] A $2$-group such that the frattini subgroup of $G$ is abelian. (cf. \cite{AT})
  \item[$5.$] An odd $p$-group whose frattini subgroup is powerful. (cf. \cite{AT})
  \item[$6.$] A $p$-group $G$ with $\e(G/Z(G)) = p$ and $p \leq 7$. (cf. \cite{AT})
 \end{itemize}
 \end{thmC}
 
 \section{Common Strategies}

In this section, we briefly explain the common strategies used to prove the Theorems in the previous section. One of the tools used is a construction introduced by R. Brown and J.-L. Loday in \cite{BL1} and \cite{BL2}, called the nonabelian tensor product $G\otimes H$. The nonabelian tensor product of groups is defined for a pair of groups that act on each other provided the actions satisfy the compatibility conditions of Definition \ref{D:1.1} below. Note that we write conjugation on the left, so $^gg'=gg'g^{-1}$ for $g,g'\in G$ and $^gg'\cdot g'^{-1}=[g,g']$ is the commutator of $g$ and $g'$.

\begin{definition}\label{D:1.1}
Let $G$ and $H$ be groups that act on themselves by conjugation and each of which acts on the other. The mutual actions are said to be compatible if
\begin{equation}
  ^{(^h g)}h_1=\; (^h(^g(^{h^{-1}}h_1)))  \;and\; ^{(^g h)}g_1=\ (^g(^h(^{g^{-1}}g_1))) \;\mbox{for \;all}\; g,g_1\in G, h,h_1\in H.
\end{equation}
\end{definition}

\begin{definition}\label{D:1.2}
If $G$ and $H$ are groups that act compatibly on each other, then the \textbf{nonabelian tensor product} $G\otimes H$ is the group generated by the symbols $g\otimes h$ for $g\in G$ and $h\in H$ with relations
\begin{equation}\label{E:1.1.1}
gg'\otimes h=(^gg'\otimes \;^gh)(g\otimes h),    
\end{equation}
\begin{equation}\label{E:1.1.2}
g\otimes hh'=(g\otimes h)(^hg\otimes \;^hh'),   
\end{equation}
\noindent for all $g,g'\in G$ and $h,h'\in H$.
\end{definition}
The nonabelian tensor square, $G\otimes G$, of a group $G$ is a special case of the nonabelian tensor product of a pair of groups $G$ and $H$, where $G=H$, and all actions are given by conjugation. 
There exists a homomorphism $\kappa : G\otimes G \rightarrow G^{\prime}$ sending $g\otimes h$ to $[g,h]$. Let $\nabla (G)$ denote the subgroup of $G\otimes G$ generated by the elements $x\otimes x$ for $x\in G$. The exterior square of $G$ is defined as $G\wedge G= (G\otimes G)/\nabla (G)$ and denote the induced homomorphism again by $\kappa : G\wedge G \rightarrow G^{\prime}$. Let $M(G):=\ker ( G\wedge G \rightarrow G' )$. It has been shown in \cite{CM} that $M(G)\cong H_{2}(G, \mathbb{Z})$. 
Now we outline most common methods to prove the previous Theorems stated in the Introduction.

\begin{itemize}
\item {\bf Strategy 1} Try to find the exponent of $G\wedge G$. If the exponent of $G\wedge G$ divides the exponent of $G$, then Schurs Conjecture holds true, since $M(G)$ is a subgroup of $G\wedge G$.

\item {\bf Strategy 2} The second strategy is almost similar to the first one. Instead of estimating the exponent of $G\wedge G$, one estimates the exponent of the commutator subgroup of a covering group of $G$, thanks to Schur (cf. \cite{GK}), the commutator subgroups of any two covering groups are isomorphic.

\item {\bf Strategy 3} One can also estimate the exponent of the Bogomolov multiplier of the group $G$.
\end{itemize}

\section{Noether's Rationality Problem and Schur's Exponent Conjecture}

Let $k$ be any field and $G$ be a finite group. Let $G$ act on the rational function field $k(x_g, g\in G)$ by $k$ automorphisms defined by $g.x_h:=x_{gh}$ for all $g,h\in G$.  Denote by $F(G)$ the fixed field  $k(x_g, g\in G)^{G}$. Noether's problem asks whether $F(G)$ is purely transcendental over $k$. Bogomolov multiplier is the group $B_0(G)=\ker\{H^2(G, \mathbb{Q}/{\mathbb{Z}})\to \bigoplus_A H^2(A,\mathbb{Q}/{\mathbb{Z}})\}$, where $A$ runs over all Abelian subgroups of $G$. Let $M_0(G)$ be the subgroup of $G\wedge G$ generated by $x\wedge y$ such that $[x,y]=1$. It is shown in \cite{PM8} that for a finite group $G$, $B_0(G)$ is non-canonically isomorphic to $M(G)/M_0(G)$. More on Bogomolov multiplier and its relation to Noether's Rationality problem can be found in \cite{B1}, \cite{B2} and \cite{S}. Bogomolov multiplier $B_0(G)$ can be seen as an obstruction to Noether's rationality problem. In particular, if $B_0(G)\neq 0$, then Noether's Rationality problem has a negative answer for $G$. In this section, we show that if the Bogomolov multiplier $B_0(G)=0$, then Schur's Conjecture \ref{C1} holds. This statement gives us a connection between Noether's Rationality problem and Schur's exponent conjecture.

Sometime in 2015, in an evening walk, Guram Donadze mentioned to the author that triviality of the Bogomolov multiplier implies Schur's conjecture. Since the author could not find a reference to this statement, we produce it here for the benefit of the reader and also because the proof is short.

\begin{lemma}\label{L:3.1}
Let $G$ be a finite group. If $B_0(G)=0$, then Conjecture $\ref{C1}$ holds true.
\end{lemma}

\begin{proof}
Note that triviality of $B_0(G)$ implies that $M(G)=M_0(G)$. Since $[x,y]=1$, we have that $x^n\wedge y=(x\wedge y)^n$. Taking $n$ to be the exponent of $G$ and noting that each $x\wedge y$ in $M_0(G)$ is contained in the center of $G\wedge G$, finishes the proof.
\end{proof}

Thus all counterexamples to Schur's conjecture will provide a negative answer to Noether's Rationality problem. Triviality of Bogomolov multiplier is not a necessary condition for Conjecture \ref{C1} to hold. Most $p$ groups of maximal class have non-trivial Bogomolov multiplier (cf. \cite{AJ}), but Conjecture \ref{C1} holds for $p$ groups of maximal class. On the other hand, the Bogomolov multiplier is trivial for all cyclic groups, but Noether's Rationality problem does not hold for all cyclic groups. (cf. \cite{RGS}, \cite{BP})

\section{Bounds depending on nilpotency class, derived length and the rank.}

Several authors have given bounds for the exponent of the Schur Multiplier in terms of the exponent of the group, the nilpotency class, derived length and the rank of the group. The authors of \cite{SN} improved the bounds given by the authors of \cite{LM} depending on the exponent and the rank of the group. Below we list some of the bounds given in terms of the nilpotency class and the derived length of a group.

\begin{thmD}
\begin{itemize}
Let $G$ be a $p$ group of nilpotency class $c$.

\item[$1.$]  $exp(M(G))\mid (exp(G))^{\ceil{\frac{c}{2}}}$. (cf. \cite{GE2})
\item[$2.$]  $exp(M(G))\mid (exp(G))^{2(\floor{\log_2 c})}$. (cf. \cite{PM1})
\item[$3.$]  $exp(M(G))\mid (exp(G))^m$, where $m=\floor{\log_{p-1} c}+1$. (cf. \cite{NS2})
\item[$4.$]  $exp(M(G))\mid (exp(G))^n$, where $n = 1+\ceil{\log_{p-1} (\frac{c+1}{p+1})}$. (cf. \cite{APT})

\end{itemize}
\end{thmD}
Note that the last bound above is an improvement of all the previous bounds.

\begin{thmE}
\begin{itemize}
Let $G$ be a solvable $p$ group of derived length $d$.

\item[$1.$]  $exp(M(G))\mid (exp(G))^d$, if $p$ is odd. (cf. \cite{NS1}, \cite{APT})
\item[$2.$]  $exp(M(G))\mid 2^{d-1}(exp(G))^{d}$, if $p=2$. (cf. \cite{NS1}, \cite{APT})
\item[$3.$]  If the exponent of $G$ is $p$, then $exp(G\wedge G) \mid (exp(G))^{d-1}$, if $p$ is odd. (cf. \cite{AT})
\item[$4.$]  If the exponent of $G$ is $2$, then $exp(G\wedge G) \mid 2^{d-2}(exp(G))^{d-1}$. (cf. \cite{AT})

\end{itemize}
\end{thmE}

Many classes of groups for which Conjecture $\ref{C1}$ is true falls under the class of Regular groups. The authors of \cite{NS2} and \cite{APT} proved the following:

\begin{thmF}
The following statements are equivalent:
\begin{itemize}
\item[$(i)$] $exp(M(G))\mid exp(G)$ for all regular $p$-groups $G$.
\item[$(ii)$] $exp(M(G))\mid exp(G)$ for all groups $G$ of exponent $p$.
\end{itemize}
\end{thmF}

By the counterexample given by Vaughan-Lee, it follows that Schur's conjecture is not true for regular $p$-groups.

\section{Interesting Questions}

The following are the questions that the author finds most interesting with regards to Schurs exponent conjecture.\\

{\bf Question 1}: Is Conjecture $\ref{C1}$ true for metabelian groups? \\

We know that Conjecture $\ref{C1}$ is false in general. But we do not have any counterexample of groups of nilpotency class less than $9$ to Conjecture $\ref{C1}$. So the following question is interesting.\\

{\bf Question 2}: Let $p$ be an odd prime. Is Conjecture $\ref{C1}$ true for $p$ groups of class less than or equal to $8$?

  \bibliographystyle{amsplain}
\bibliography{Bibliography}
\end{document}